\documentclass{tran-l}
\usepackage{tikz}

 \newtheorem{thm}{Theorem}[section]
 
 \newtheorem{lem}[thm]{Lemma}
 \newtheorem{prop}[thm]{Proposition}
 \theoremstyle{definition}
 \newtheorem{defn}[thm]{Definition}
 \theoremstyle{remark}
 \newtheorem{rem}[thm]{Remark}
 \newtheorem{exmp}[thm]{Example}
 \numberwithin{equation}{section}

\begin{document}
\title[Maximum Independent Set of Clique]
 {Maximum Independent Set of Cliques and The Generalized Mantel's Theorem}

\author{Hossein Teimoori Faal}

\address{Department of Mathematics and Computer Science, Allameh Tabataba'i University, Tehran, Iran}

\email{Hossein.teimoori@atu.ac.ir}



\date{}


\commby{}

\begin{abstract}

A complete subgraph of any simple graph $G$ on $k$
vertices is called a $k$-\emph{clique} of $G$.  
In this paper, we first introduce the concept of 
the value of a $k$-clique ($k>1$) as an extension of the 
idea of the degree of a given vertex. Then, we obtain 
the generalized version of handshaking lemma which
we call it clique handshaking lemma. The well-known
classical result of Mantel states that the maximum number 
of edges in the class of triangle-free graphs with $n$
vertices is equal to $\frac{n^{2}}{4}$. Our main goal 
here is to find an extension of the above result 
for the class of $K_{\omega+1}$-free graphs, using the 
ideas of the value of cliques and the clique handshaking lemma.

\end{abstract}
\maketitle

\section{INTRODUCTION}

Finding the \emph{maximum} values of some key \emph{invariants} in discrete structures with 
\emph{forbidden} (finite) family 
of substructures is an interesting problem 
in the area of \emph{extremal combinatorics} 
with potential applications in 
theoretical and applied computer science.  
One of the classical problems of these kind 
is the well-known \emph{Mantel's} theorem \cite{Mant(1907)}
which answers the question about the maximum number of 
edges in any simple graph in which the family 
of forbidden subgraphs consists of only the triangle 
graph $K_{3}$. 
There are many interesting proofs of 
the well-known Mantel's theorem 
and one of those beautiful proofs are based on the idea of 
\emph{maximum} independent set of vertices.  
Roughly speaking, the basic idea is to partition the vertex set of a given graph $G=(V,E)$ into two sets 
$A$ and $B$. The first set $A$ is an independent set of 
maximum size (maximum number of vertices) and 
$B$ is the rest of vertices. Now, considering the maximality of $A$, the \emph{triangle-freeness} 
of $G$ and the well-known \emph{handshaking} 
lemma, one can bound (an upper bound) the number of edges 
based on the sum of degrees of vertices lie 
in $B$. Finally combining all previous findings 
with the well-known 
\emph{arithmetic-geometric mean} inequality 
we get the classical Mantel's theorem.    
\\
It seems that we can extend the idea of the independent set of vertices (maximum independent set)
to the independent set of edges (maximum matching). 
Then, using all the previous machinery, one can get 
a generalization of Mantel's result for the 
class of $K_{4}$-free graphs which we call it 
\emph{Edge Mantel's theorem}.  
Next, we generalize the concept of the degree of 
a vertex to a higher $k-$clique ($k>1$) by introducing 
the idea of the \emph{value of a clique}. This simply 
means that a value of a clique can be defined as
the number of common neighbors of it's vertices.
In this direction, we also obtain a higher clique 
generalization of the handshaking lemma which we 
call it \emph{clique handshaking lemma}. 
Finally, using the same machinery introduced for 
proving the classical Mantel's theorem, we obtain 
the \emph{Clique Mantel's theorem}.

\section{Basic Definitions and Notations}

Throughout this paper, we will assume that our graphs are 
finite, simple and undirected. For terminologies 
which are not defined here, one can refer to the book
\cite{West(2001)}. 
\\
For a give graph $G=(V,E)$, the vertex set and the edge set will be denoted by $V(G)$ and $E(G)$, respectively. 
For a vertex $v \in V(G)$, it's \emph{open neighborhood} 
denoted by $N_{G}(v)$
is the set of vertices \emph{adjacent} to $v$. 
A subgraph of $G$ consisting of all those
vertices that are pairwise adjacent is called 
a \emph{complete} subgraph (clique) of $G$.
A complete subgraph with $k$ vertices will 
be called a $k$-\emph{clique}. The set 
of all $k$-cliques in $G$ is denoted by $\Delta_{k}(G)$. 
We will also denote the number of $k$-cliques of a graph 
$G$ by $c_{k}(G)$. 
A complete subgraph on three vertices is called
a \emph{triangle}. A subset of vertices with
no edges among them is called an \emph{independent} set 
of $G$.   
\\
A generalization of the concept of the \emph{degree} of a vertex can be extended to the  
\emph{edge value}, as follows.  

\begin{defn}
	For a given graph $G=(V,E)$ and an edge $e=\{u,v\} \in E(G)$, the value 
	of $e$ denoted by $val_{G}(e)$ is defined as the number of common neighbors of two 
	end vertices $u$ and $v$ of the edge $e$. 
	More precisely, we have 
	$$
	val_{G}(e) = 
	\Big\vert N_{G}(u) \cap N_{G}(v) \Big\vert.  
	$$
\end{defn}

\begin{rem}
	It is interesting to note that 
	a definition similar to the edge value has been 
	given in the literature \cite{PavNest(2004)}. Indeed, the \emph{co-degree} of 
	two vertice $u, v \in V(G)$, not necessarily adjacent, 
	is defined as the number of their common neighborhoods. 
\end{rem}

\begin{exmp}
	Let $G=(V,E)$ be a graph depicted in Figure \ref{fig:edgeval}. 
	Then, the values of edges are, as follows. 
	
	\begin{equation}
	val_{G}(e_{12})=val_{G}(e_{13})=val_{G}(e_{23})
	=1,\hspace{0.4cm} val_{G}(e_{34})=0.
	\end{equation}	
\end{exmp}

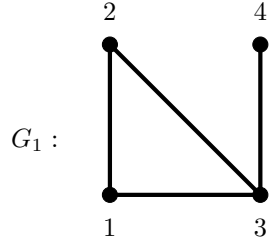
\begin{figure}
	\begin{center}
		\begin{tikzpicture}
		\draw [ultra thick](-6,0)-- (-6,2) --(-4,0) -- (-6,0);
		\draw [ultra thick](-4,0)-- (-4,2);
		\draw [fill] (-6,0) circle [radius=0.1];
		\draw [fill] (-6,2) circle [radius=0.1];
		\draw [fill] (-4,0) circle [radius=0.1];
		\draw [fill] (-4,2) circle [radius=0.1];
		\node [below] at (-7,1) {$G_{1}:$};
		\node [below] at (-6,-0.20) {$1$};
		\node [above] at (-6,2.2) {$2$};
		\node [below] at (-4,-0.20) {$3$};
		\node [above] at (-4,2.2) {$4$};
		\end{tikzpicture}	
		\caption{The values of edges for the graph $G_{1}$}\label{fig:edgeval}
	\end{center}
\end{figure}

Next, we generalize the above idea for any $k$-clique $q_{k} \in \Delta_{k}(G)$ $(k>1)$ of $G$. 

\begin{defn}
	[Value of a Clique]
	Let $G=(V,E)$ be a simple graph and $q_{k}$ be a $k$-clique of $G$. Then, we define 
	the \emph{value} of $q_{k}$ denoted by $val_{G}(q_{k})$, as follows 
	\begin{equation}
	val_{G}(q_{k}) = 
	\Big\vert 
	\bigcap_{v \in V(q_{k})} N_{G}(v)  
	\Big\vert.  
	\end{equation}
	
\end{defn} 

As an extension of the well-known \emph{handshaking lemma}, we have the following key result.

\begin{lem}
	[Clique Handskaing Lemma]
	
	For a simple graph $G=(V,E)$, we have 
	
	\begin{equation}\label{cliqhand1}
	\sum_{q_{k} \in \Delta_{k}(G)} val_{G}
	(q_{k}) = (k+1) c_{k+1}(G), \hspace{0.4cm}
	(k \geq 1).  
	\end{equation}

\end{lem}

One can given several proofs of the above lemma. Here, we 
present a proof which based on the idea of 
\emph{double counting}. 

\begin{proof}
	Let $G=(V,E)$ be any simple graph. Define the set 
	$I_{k}(G)$, as follows. 
	
	\begin{equation}
	I_{k}(G)
	=
	\{ 
	(q_{k}, q_{k+1}) \in \Delta_{k}(G) \times 
	\Delta_{k+1}(G)
	~\vert~ q_{k}~\text{is a subgraph of}~q_{k+1}
	\}.
	\end{equation}	
	The proof proceeds by counting the set $I_{k}(G)$
	in two ways. 
	\\
	{\bf Case I.} We first fix the clique $q_{k}$. 
	Then, it is clear that the number of such $(k+1)$-cliques 
	is exactly $val_{G}(q_{k})$. Now, summing over all those $k$-cliques $q_{k}$ will result in 
	$
	\sum_{q_{k} \in \Delta_{k}(G)} val_{G}
	(q_{k})
	$	
	.
	\\
	{\bf Case II.}	
	Next, we fix the $(k+1)$-clique $q_{k+1}$. Then, it is 
	obvious that the number of such $k$-cliques 
	which are the subgraph of $q_{k+1}$ is equal to 
	$k+1$. Thus, by summing over all $(k+1)$-cliques,
	we get $(k+1)c_{k+1}(G)$. 
	\\
	Finally, the proof is complete by the 
	\emph{double-counting} technique.  
\end{proof}

\begin{rem}
	In is worthy to note that the above lemma is 
	called the \emph{transfer equations} by Knill in \cite{PavNest(2004)} which is even true 
	for the generalized discrete structures like 
	\emph{simplicial complexes}. The transfer equations are used 
	to obtain a \emph{graph-theoretical} version of 
	the well-known \emph{Gauss-Bonnet} formula. 
\end{rem}

\section{Main Results}

In this section, we will use the idea of 
the clique value to generalize the following
clique-counting inequality due to Mantel
\cite{Mant(1907)}.

\begin{thm}
	[
	Mantel's Theorem for Triangle-free graphs
	]
	For a given triangle-free graph
	$G=(V,E)$ with $n$ vertices and 
	$m$ edges, we have
	\begin{equation*}
	m 
	\leq 
	\frac{n^2}{4}.
	\end{equation*}
\end{thm} 

The motivation of this paper originates from 
the proof of the above classical result which is based 
on the idea of \emph{maximality}. Thus, we 
also include the proof. 
\begin{proof}
	Let $A \subseteq V(G)$ be an independent set of 
	maximum size (a maximum independent set). Next, we 
	put 
	$
	B=V(G)-A
	$
	.
	Since $G$ is triangle-free, the open neighborhood of any 
	arbitrary vertex  
	$v \in V(G)$ is an independent set. Hence, 
	by the maximality of $A$, we immediately conclude 
	that
	
	\begin{equation}\label{keyind1}
	\deg_{G}(v) = \vert N_{G}(v) \vert  
	\leq \vert A \vert, \hspace{0.5cm}\forall~v\in V(G). 
	\end{equation}\label{keyind2}
	On the other hand, since $A$ is an independent set of vertices,
	we obviously have 
	\begin{equation}\label{keyident1}
	\sum_{v \in A}\deg_{G}(v) \leq m. 
	\end{equation} 
	
	Considering the well-known handshaking lemma, we also 
	get 
	\begin{equation}\label{keyident2}
	\sum_{v \in A}\deg_{G}(v) + \sum_{v \in B}\deg_{G}(v)
	= 2m. 
	\end{equation}
	
	Form identities (\ref{keyident1}) and (\ref{keyident2}),
	we conclude that
	\begin{equation}\label{keyinequalt2}
	m \leq \sum_{v \in B}\deg_{G}(v). 
	\end{equation}
	
	Thus, from relations (\ref{keyind1}),
	(\ref{keyinequalt2}) and 
	the \emph{arithmetic-geometric mean} 
	inequality, we finally obtain 
	
	\begin{eqnarray}
	m & \leq &  \sum_{v \in B}\deg_{G}(v)\nonumber\\
	& \leq & \sum_{v \in B} \vert A \vert\nonumber\\
	& = & \vert A \vert \vert B \vert\nonumber\\
	& \leq & \Big(\frac{\vert A \vert 
		+ \vert B \vert}{2} \Big)^{2} \nonumber\\
	& \leq & \frac{n^{2}}{4},\nonumber 
	\end{eqnarray}
	as required. 
	
\end{proof}

\begin{rem}
	It is important to note that from the 
	arithmetic-geometric mean inequality in the above proof
	it immediately follows that the \emph{extremal graph} 
	for Mantel's classical result is the balanced 
	complete bipartite graph $K_{\frac{n}{2},\frac{n}{2}}$.
	But here, we are only interested in extremal bounds
	(inequalities) and 
	not the extremal graphs themselves. 
\end{rem}

The next result is a slight generalization of the
Mantel's theorem and is based on the idea of 
the edge values and the \emph{maximum matchings} 
in graphs.

\begin{thm}\label{edgeMant1}
	[
	Edge Mantel's Theorem
	]
	For a given 
	$
	K_{4}
	$
	-free graph
	$G=(V,E)$ with $m$ edges and 
	$t$ triangles, we have
	\begin{equation*}
	t
	\leq 
	\frac{m^2}{8}.
	\end{equation*}
\end{thm}

\begin{proof}
	Let
	$A$ be the 
	maximum independent set of 
	edges ( the maximum matching ); that is, 
	the edges with no common end points. 
	.
	Put 
	$
	B = E(G)
	- A 
	$.
	Since 
	$G$ is a 
	$
	K_{4}
	$
	-free graph, $N_{G}(u) \cap N_{G}(v)$ is an independent set of edges for any edge $e=uv \in E(G)$.  
	Therefore, by maximality of $A$,
	we have 
	\begin{equation}\label{keyinq1}
	val_{G}(e) = \vert N_{G}(u) \cap N_{G}(v)  \vert 
	\leq \vert A \vert.
	\end{equation}
	Next, we observe that
	since $A$ is an independent set of edges, we clearly
	have 
	\begin{equation}\label{keyinequt1}
	\sum_{e \in A}val_{G}
	(e) \leq 
	t.
	\end{equation}
	On the other hand, considering the clique handshaking 
	lemma (\ref{cliqhand1}) for $\omega(G)=3$ , we have
	
	\begin{equation}\label{trihand}
	\sum_{e \in A}val_{G}(e) + 
	\sum_{e \in B}val_{G}(e) = 
	3t. 
	\end{equation}
	
	Now, form formulas (\ref{keyinequt1})
	and (\ref{trihand}), we conclude that 
	
	\begin{equation}\label{mainIneq}
	2t \leq \sum_{e \in B}val_{G}(e). 
	\end{equation}

	Finally considering the arithmetic-geometric mean 
	inequality and the inequalities 
	(\ref{keyinq1}) and (\ref{mainIneq}), we  
	get
	\begin{eqnarray*}
		t 
		& \leq & 
		\frac{1}{2}\sum
		_
		{
			e \in B
		} 
		val_{G}(e) \nonumber\\
		& \leq & 
		\frac{1}{2}\sum
		_
		{
			e \in B
		} 
		\vert
		A
		\vert = \frac{1}{2}
		\vert
		A
		\vert
		\vert
		B
		\vert
		\nonumber\\
		& \leq &
		\frac{1}{2}\left(
		\frac{
			\vert
			A
			\vert
			+
			\vert
			B
			\vert
		}{2}
		\right)
		^{2}\nonumber\\
		& = & 
		\frac{m^2}{8}, 
	\end{eqnarray*}
	
	which completes the proof.

\end{proof}

Now, considering the idea of the value
of a clique and the clique handshaking lemma using similar arguments as above, 
we obtain the following generalization of Theorem \ref{edgeMant1}. 

\begin{thm}\label{MainThm1}
	[
	Clique Mantel's Theorem
	]
	For a given 
	$
	K_{\omega + 1}
	$
	-free graph
	$G=(V,E)$ with $c_{\omega-1}(G)$ cliques of 
	size $\omega(G) -1$
	and 
	$c_{\omega}(G)$ cliques
	of size 
	$
	\omega(G)
	$, we have
	\begin{equation*}
	c_{\omega}(G)
	\leq 
	\frac{1}{\omega(G)-1}.
	\frac{c^{2}_{\omega-1}(G)}{4}.
	\end{equation*}
	
\end{thm}

\section{Concluding Remarks and Future Works}

In this paper, we obtain an upper bound for 
the number $k$-clique in the class of $(k+1)$-cliques-free 
graphs; that is, the class of those graphs not containing
any complete subgraph of on $k+1$ vertices. 
The basic ideas were maximality, clique handshaking 
identity and using 
the arithmetic-geometric mean inequality. 
\\
Our future project is to consider a more general class 
of graphs that we will call them 
$\mathcal{H}$-free graphs. 
We recall that the \emph{increasing family} \cite{BresarImrichKlavzar(1998)} of 
graphs $\mathcal{H}$ 
is the following 
$$
\mathcal{H}= 
\{
H_{1}, H_{2}, \ldots, H_{k}, H_{k+1}, \ldots 
\},
$$
in which $H_{1}= K_{1}$ and each $H_{i}$
is an \emph{induced subgraph} of $H_{i+1}$, for all $i$. 
Our main \emph{goal} is to find an upper bound 
similar to that of Theorem \ref{MainThm1} for the maximum number 
of copies of $H_{k}$ in the class of those graphs 
not containing any subgraph \emph{isomorphic} 
to $H_{k+1}$ (for any integer $k>1$). 
To achieve this goal, we need two main steps.
We have to first define a similar notion of the value 
of a clique for any graph $H_{k}$ in $\mathcal{H}$.
Then, we need to find an analogue of our key lemma; 
the clique handshaking lemma. 
We will call this $\mathcal{H}$-handshaking lemma. 
The following result, due to $Kelly$ \cite{Tutte(1979)}, will play 
an essential role. 

\begin{prop}\label{Kellyprop1}
	Let $G=(V,E)$ be an $n$-vertex graph with no 
	\emph{isolated} vertices. 
	Then for any graph $H$ on $k$ vertices, we have 
	$$
	(n-k)s(H,G) = \sum_{v \in V} s(H,G-v),
	$$ 
	where $s(H,G)$ denotes the number of subgraphs 
	of $G$ isomorphic to $H$. 
\end{prop}

Note that in particular case where $H_{k}$ is a 
$k$-clique, it is not hard to show that Proposition \ref{Kellyprop1} is equivalent to our 
clique handshaking lemma.  
We can also recover all the results of 
this paper in this special case. 
More details will appear in our 
sequel paper \cite{PavNest(2004)}.

\end{document}